\documentclass[psamsfonts,fceqn,leqno]{amsart}
\usepackage{mathrsfs,latexsym,amsfonts,amssymb,curves,epic}
\usepackage{color}

\newtheorem{theorem}{Theorem}[section]
\newtheorem{corollary}[theorem]{Corollary}
\newtheorem{proposition}[theorem]{Proposition}
\newtheorem{lemma}[theorem]{Lemma}
\newtheorem{question}[theorem]{Question}
\newtheorem{problem}[theorem]{Problem}
\theoremstyle{definition}
\newtheorem{definition}[theorem]{Definition}

\newtheorem{example}[theorem]{Example}

\begin{document}
\title[Transversal, $T_{1}$-independent, and $T_{1}$-complementary  paratopological group topologies]
{Transversal, $T_{1}$-independent, and $T_{1}$-complementary  paratopological group topologies}

  \author{Fucai Lin}
  \address{(Fucai Lin): School of mathematics and statistics,
  Minnan Normal University, Zhangzhou 363000, P. R. China}
  \email{linfucai2008@aliyun.com; linfucai@mnnu.edu.cn}

  \thanks{The author is supported by the NSFC (No. 11571158), the Natural Science Foundation of Fujian Province (No. 2017J01405) of China, the Program for New Century Excellent Talents in Fujian Province University, the Project for Education Reform of Fujian Education Department (No. FBJG20170182), the Institute of Meteorological Big Data-Digital Fujian and Fujian Key Laboratory of Data Science and Statistics.}


  \begin{abstract}
We discuss the class of paratopological groups which admits a transversal, $T_{1}$-independent and $T_{1}$-complementary paratopological group topology. We show that the Sorgenfrey line does not admit a $T_{1}$-complementary Hausdorff paratopological group topology, which gives a negative answer to \cite[Problem 10]{AT2017}. We give a very useful criterion for transversality in term of submaximal paratopological group topology, and prove that if a non-discrete paratopological group topology $G$ contains a central subgroup which admits a transversal paratopological group topology, then so does $G$.  We introduce the concept of $PT$-sequence and give a characterization of an Abelian paratopological group being determined by a $PT$-sequence. As the applications, we prove that the Abelian paratopological group, which is endowed with the strongest paratopological group topology being determined by a $T$-sequence, does not admit a $T_{1}$-complementary Hausdorff paratopological group topology on $G$. Finally, we study the class of countable paratopological groups which is determined by a $PT$-filter, and obtain a sufficient condition for a countable paratopological group $G$ being determined by a $PT$-sequence which admits a transversal paratopological group topology on $G$ being determined by a $PT$-sequence.
  \end{abstract}

  \maketitle
\section{Introduction}
In 1936, Birkhoff has begun to study the lattice of all topologies on a given set in \cite{Bi1936}, which was published in Fundamenta Mathematicae. Birkhoff was probably the first in recognizing the importance of the structure
of the set $\mathcal{L}(X)$ all topologies on a given set $X$ being ordered by inclusion. Two distinct topologies on a set $X$ are said to be complementary if their intersection is the indiscrete topology and their
supremum is the discrete topology. The complementarity of the lattice $\mathcal{L}(X)$ was
established independently by Steiner \cite{St1966} and van Rooij \cite{RA1968}. Moreover, Bagley \cite{Ba1955} and Steiner \cite{St1966} studied the lattice $\mathcal{L}_{1}(X)$ of all $T_{1}$-topologies
on a set $X$ and the $T_{1}$-complements, where two $T_{1}$ topologies $\tau_{1}$ and $\tau_{2}$ on a set $X$ are called $T_{1}$-complementary if the intersection $\tau_{1}\cap \tau_{2}$ is the
cofinite topology and their supremum is the discrete topology on $X$. In \cite{DTY2005}, D. Dikranjan, M. Tkachenko, I. Yaschenko denote $\mathcal{L}(G)$ by all Hausdorff group topologies on a group $G$: they proved that no topology $\tau\in\mathcal{L}(G)$ admits a $T_{1}$-complement. Therefore, they split the notion of $T_{1}$-complementarity
into two natural components: transversality and $T_{1}$-independence.
Indeed, the investigation of $T_{1}$-independent and transversal group topologies on Abelian groups was initiated in \cite{DTY2005}, \cite{MY2002} and \cite{ZP2001}.

\begin{definition}\cite{DTY2005}
Two non-discrete topologies $\tau_{1}$ and $\tau_{2}$ on a set $X$ are called {\it transversal} if
their union generates the discrete topology on $X$.
\end{definition}

\begin{definition}\cite{DTY2005}
Two $T_{1}$ topologies $\tau_{1}$ and $\tau_{2}$ on a set $X$ are called {\it $T_{1}$-independent} if their
intersection $\tau_{1}\cap \tau_{2}$ is the cofinite topology on $X$.
\end{definition}

Recall that a {\it paratopological group} is a group with topology in which multiplication is jointly continuous. If $G$ is a paratopological group  and the inverse operation of $G$ is continuous, then $G$ is called a {\it topological group}. It is well known that there exists a paratopological group which is not a
topological group; Sorgenfrey line (\cite[Example 1.2.2]{E1989}) is such an example. In the past 30 years, the theory of paratopological groups has become a hot research topic and obtained a lot of interesting results, see \cite{MT2014}. In this paper, we discuss the the transversality, $T_{1}$-independence and $T_{1}$-complementarity in the class of paratopological groups.

There is a big difference between $T_{1}$-independence and transversality of paratopological group topologies. Indeed, the union $\tau_{1}\cup \tau_{2}$ of two paratopological group topologies on $G$ always generates a paratopological group topology, but the intersection $\tau_{1}\cap \tau_{2}$ can fail to be a paratopological group topology. Here we study the class of paratopological groups that admit transversal, $T_{1}$-independent and $T_{1}$-complementary paratopological group topologies respectively. The paper is organized as follows.

In Section 3 we discuss the existence of a pair of $T_{1}$-complementary Hausdorff paratopological group topologies on an infinite group $G$. We introduce the concept of $SD$-paratopological group, and then prove that, for any an infinite Hausdorff $SD$-paratopological group, there does not exist a pair of $T_{1}$-complementary Hausdorff paratopological group topologies (Theorem~\ref{t0}); in particular, since the Sorgenfrey line $\mathbb{S}$ is a $SD$-paratopological group, there does not admit a $T_{1}$-complementary Hausdorff paratopological group topology, which gives a negative answer to Problem~\ref{q1}. Further, we discuss the existence of a pair of $T_{1}$-complementary Hausdorff paratopological group topologies for any a countable infinite group, and prove that for a countable infinite group $G$, it does not admit a pair of $T_{1}$-complementary Hausdorff paratopological group topologies if and only if it does not admit $T_{1}$-independent Hausdorff paratopological group topologies (Theorem~\ref{t2}).

In Section 4 we introduce the concept of submaximal paratopological group, and develop the technique of submaximal paratopological group topology. We prove that a submaximal paratopological group on a group is a topological group (Theorem~\ref{tttt}), and then we give a very useful criterion for the transversality in term of submaximal paratopological group topology (Theorem~\ref{tttt1}). As an application, we prove in Theorem~\ref{t03} that if a non-discrete paratopological group topology $G$ contains a central subgroup which admits a transversal paratopological group topology, then so does $G$. Moreover, we discuss when a submaximal paratopological group is precompact. We obtain that a submaximal Abelian paratopological group is precompact if and only if each non-precompact paratopological group topology admits a transversal paratopological group topology (Theorem~\ref{tt00}).

In Section 5 we study the class of Abelian paratopological groups which is determined by a $PT$-sequence. We give a characterization of an Abelian paratopological group being determined by a $PT$-sequence (Theorem~\ref{th003}), and show that the family $\sum(\varphi^{\ast})$ is a base of neighborhoods of $0$ for the paratopological group $P(G|\varphi)$ for each $PT$-filter $\varphi$ on $G$ (Theorem~\ref{t00}). As the applications, we first prove that the Abelian paratopological group, which is endowed with the strongest paratopological group topology being determined by a $PT$-sequence, is a sequential non-Fr\'{e}chet-Urysohn space (Theorem~\ref{t001} and Corollary~\ref{c001}); then we show that the Abelian paratopological group $G$, which is endowed with the strongest paratopological group topology being determined by a $T$-sequence, does not admit a $T_{1}$-complementary Hausdorff paratopological group topology on $G$ (Theorem~\ref{t007}).

In Section 6 we main discuss the class of countable paratopological groups which is determined by a $PT$-filter (in particular, $PT$-sequence). We obtain a characterization of a countable paratopological group being determined by a $PT$-filter (Theorem~\ref{t009}). We also prove that a countable group $G$ admits a non-discrete $T_{1}$-paratopological group topology if and only if there exits a non-trivial $PT$-sequence of elements of $G$ (Theorem~\ref{t005}). Finally, we show that, for a countable paratopological group $G$ being determined by a $PT$-sequence $\{a_{n}\}_{n\in\omega}$, if $P(G|\{a_{n}\})$ is Hausdorff, then there exists a nontrivial $PT$-sequence $\{b_{n}\}_{n\in\omega}$ in $G$ such that $P(G|\{a_{n}\})$ and $P(G|\{b_{n}\})$ are transversal (Theorem~\ref{t009}).

  \maketitle
\section{Notation and terminology}
We denote by $\omega$ and $\mathbb{Z}$ the sets of naturals and integers, respectively, by $\mathbb{R}$ the reals, by $\mathbb{Q}$ the rationals, by $\mathbb{T}$ the unit circle group in the complex plain
$\mathbb{C}$. For undefined
  notation and terminology, the reader may refer to \cite{AT2008} and
  \cite{E1989}.

Let $(G, \tau)$ be a (para)topological group with identity $e$ and $H\subset G$, and denote by $\tau(e)$ the set of all neighborhoods of $e$ in $G$. The subgroup $c_{G}(H)$ denotes the centralizer of $H$ in $G$, i.e., the maximal subgroup of $G$ all elements of which commute with the elements of $H$. The (para)topological group $G$ is called {\it precompact} (resp., {\it $\omega$-narrow}) if $G$ can be covered by finitely (resp., countably) many translates of each neighborhood of $e$.
A paratopological group $G$ is {\it saturated} if for any neighborhood $U$ of $e$ the set $U^{-1}$ has nonempty
interior in $G$.
Call an infinite group $G$
without non-discrete Hausdorff group topologies a {\it Markov group}.

Recall that the {\it Sorgenfrey line} is a topology on $\mathbb{R}$ whose base consists of half-open intervals $[a, b)$, with $a<b$.

Let $X$ be a space. A subset $P$ of $X$ is called a
  {\it sequential neighborhood} of $x \in X$, if each
  sequence converging to $x$ is eventually in $P$. A subset $U$ of
  $X$ is called {\it sequentially open} if $U$ is a sequential neighborhood of
  each of its points. A subset $F$ of
  $X$ is called {\it sequentially closed} if $X\setminus F$ is sequentially open. The space $X$ is called a {\it sequential space} if each
  sequentially open subset of $X$ is open. The space $X$ is said to be {\it Fr\'{e}chet-Urysohn} if, for
each $x\in \overline{A}\subset X$, there exists a sequence
$\{x_{n}\}$ in $A$ such that $\{x_{n}\}$ converges to $x$.

A space $X$ is called an {\it $S_{2}$-space} ({\it Arens' space})  if
$$X=\{\infty\}\cup \{x_{n}: n\in \mathbb{N}\}\cup\{x_{n, m}: m, n\in
\omega\}$$ and the topology is defined as follows: Each
$x_{n, m}$ is isolated; a basic neighborhood of $x_{n}$ is
$\{x_{n}\}\cup\{x_{n, m}: m>k\}$, where $k\in\omega$;
a basic neighborhood of $\infty$ is $$\{\infty\}\cup (\bigcup\{V_{n}:
n>k\})\ \mbox{for some}\ k\in \omega,$$ where $V_{n}$ is a
neighborhood of $x_{n}$ for each $n\in\omega$.

For each $n\in\omega$, let $S_{n}$ be a sequence converging to
  $x_{n}\not\in S_{n}$. Let $T=\bigoplus_{n\in\omega}(S_{n}\cup\{x_{n}\})$ be the topological sum of $\{S_{n}
  \cup \{x_{n}\}: n\in\omega\}$. Then
  $S_{\omega}=\{x\}  \cup \bigcup_{n\in\omega}S_{n}$
  is the quotient space obtained from $T$ by
  identifying all the points $x_{n}\in T$ to the point $x$. The space $S_{\omega}$ is called a {\it sequential fan}.

  \maketitle
\section{The $T_{1}$-complementarity of paratopological groups}
In \cite{AT2017},  A. B{\l}aszczyk and M. Tkachenko gave a survey on the recent advance in the theory of $T_{1}$-complementarity on a sets and groups, and posed the following two pen problems.

\begin{problem}\cite[Problem 9]{AT2017}\label{q0}
Does there exist a pair of $T_{1}$-complementary Hausdorff paratopological group topologies on an infinite group $G$?
\end{problem}

\begin{problem}\cite[Problem 10]{AT2017}\label{q1}
Does the Sorgenfrey line $\mathbb{S}$ adimit a $T_{1}$-complementary Hausdorff paratopological group topology?
\end{problem}

In this section, we main discuss the $T_{1}$-complementarity of paratopological groups, and prove that, for any infinite Hausdorff saturated paratopological group $(G, \tau)$, there does not exist a Hausdorff paratopological group topology $\sigma$ on $G$ such that $\sigma$ and $\tau$ are $T_{1}$-complementary, which gives a partial answer to Problem~\ref{q0}; in particular, the Sorgenfrey line $\mathbb{S}$ does not adimit a $T_{1}$-complementary Hausdorff paratopological group topology, which gives a negative answer to Problem~\ref{q1}. We first define a new concept.

A paratopological group $(G, \tau)$ is called a {\it $SD$-paratopological group} if there exists a saturated paratopological group topology $\gamma$ weaker than $\tau$ such that each dense subset in $(G, \gamma)$ is dense in $(G, \tau)$. Clearly, each saturated paratopological group is a $SD$-paratopological group.

Now we can prove one of our main results in this section.

\begin{theorem}\label{t0}
For any an infinite Hausdorff $SD$-paratopological group $(G, \tau)$, there does not exist a Hausdorff paratopological group topology $\sigma$ on $G$ such that $\sigma$ and $\tau$ are $T_{1}$-complementary.
\end{theorem}

\begin{proof}
Suppose not, assume that $G$ admits a $T_{1}$-complementary Hausdorff paratopological group topology $\sigma$ such that $\sigma$ and $\tau$ are $T_{1}$-complementary. Since $(G, \tau)$ is a $SD$-paratopological group, there exists a Hausdorff saturated paratopological group topology $\gamma$ weaker than $\tau$ such that each dense subset in $(G, \gamma)$ is dense in $(G, \tau)$.
We claim that $U$ is dense in the topology $\tau$ for any $U\in\sigma(e)$. Indeed, since $(G, \tau)$ is a $SD$-paratopological group, it suffices to prove that $U$ is dense in the topology $\gamma$ for any $U\in\sigma(e)$.

Suppose not, there exist $g\in G$ and $W\in \gamma(e)$ such that $gW^{2}\cap U=\emptyset$, hence $gW\cap UW^{-1}=\emptyset$. Since $(G, \gamma)$ is a saturated paratopological group, the interior of $W^{-1}$ in $\gamma$ is non-empty. Let $O=\mbox{int}_{\gamma}(W^{-1})$. Clearly, $O$ is open in $\tau$ and $gW\cap UO=\emptyset$, which shows that $G\setminus UO$ is infinite. However, since $\sigma$ and $\tau$ are $T_{1}$-independent, $G\setminus UO$ must be finite, which is a contradiction. Therefore, $U$ is dense in the topology $\gamma$ for any $U\in\sigma(e)$.

Therefore, $V\cap U$ is infinite for any $U\in\sigma(e)$ and $V\in\tau(e)$.
However, since $\tau$ and $\sigma$ are transveral, there exist $U\in \sigma(e)$ and $V\in\sigma (e)$ such that $U\cap V=\{0\}$, which is a contradiction. Therefore, there does not exits a Hausdorff paratopological group topology $\sigma$ on $G$ such that $\sigma$ and $\tau$ are $T_{1}$-complementary.
\end{proof}

Obviously, the Sorgenfrey line $\mathbb{S}$ is a saturated paratopological group and each saturated paratopological group is a $SD$-paratopological group. Hence we have the following corollary by Theorem~\ref{t0}, which gives a negative answer to Problem~\ref{q1}.

\begin{corollary}\label{cooo}
For any an infinite Hausdorff saturated paratopological group $(G, \tau)$, there does not exist a Hausdorff paratopological group topology $\sigma$ on $G$ such that $\sigma$ and $\tau$ are $T_{1}$-complementary; in particular, the Sorgenfrey line $\mathbb{S}$ does not adimit a $T_{1}$-complementary Hausdorff paratopological group topology.
\end{corollary}

By the proof of Theorem~\ref{t0}, it easily prove the following theorem.

\begin{theorem}\label{t1}
For any an infinite $SD$-paratopological group $(G, \tau)$ being $T_{0}$, there does not exist a $T_{0}$-paratopological group topology $\sigma$ on $G$ such that $\sigma$ and $\tau$ are complementary.
\end{theorem}

Next, we prove that for a countable infinite group, the existence of $T_{1}$-complementary Hausdorff paratopological group topologies is equivalent to the existence of $T_{1}$-independent Hausdorff paratopological group topologies.

\begin{theorem}\label{t2}
For a countable infinite group $G$, it does not admit a pair of transversal $T_{1}$-independent Hausdorff paratopological group topologies if and only if it does not admit a pair of $T_{1}$-independent Hausdorff paratopological group topologies.
\end{theorem}

\begin{proof}
Clearly, it suffices to prove the necessity. Assume that $G$ does not admit a pair of transversal $T_{1}$-independent Hausdorff paratopological group topologies, and assume that $\tau_{1}$ and $\tau_{2}$ are $T_{1}$-independent Hausdorff paratopological group topologies on a countable infinite group $G$. By Shakhmatov's result in \cite{Sh1984} (see also \cite[Problem 5.2.A]{AT2008}), for each $i=1, 2$, there exists a Hausdorff paratopological group topology $\sigma_{i}\subset \tau_{i}$ on $G$ such that $w(G, \sigma_{i})=\omega$. We claim that $\sigma_{1}$ and $\sigma_{2}$ are also $T_{1}$-independent. Indeed, let $\mathscr{A}$ be the family all neighborhoods at $e$ in $\sigma_{1}\wedge \sigma_{2}$. Then it suffices to prove $\bigcap\mathscr{A}=\{e\}$. Since $w(G, \sigma_{i})=\omega$ for $i=1, 2$, there exist two decreasing sequences of subsets $\{U_{n}: n\in\mathbb{N}\}$ and $\{V_{n}: n\in\mathbb{N}\}$ of $G$ such that $\{U_{n}: n\in\mathbb{N}\}$ and $\{V_{n}: n\in\mathbb{N}\}$ are open neighborhood bases of $\sigma_{1}$ and $\sigma_{2}$ at $e$ respectively. Take an arbitrary $x\in\bigcap\mathscr{A}$. Then it is easy to see that $x\in\bigcap\mathscr{A}\subset \bigcap\{U_{n}V_{n}: n\in\mathbb{\mathbb{N}}\}$. Then for each $n\in\mathbb{N}$ there exist $y_{n}\in U_{n}$ and $z_{n}\in V_{n}$ such that $x=y_{n}z_{n}$. Clearly, $\{y_{n}\}$ and $\{z_{n}\}$ converge to $e$ in  $\sigma_{1}$ and $\sigma_{2}$ respectively, hence $\{y_{n}z_{n}\}$ converges to $e$ in $\sigma_{1}\wedge \sigma_{2}$. Thus $x=e$.

By the assumption, it follows that the union $\sigma_{1}\cup\sigma_{2}$ generates a non-discrete Hausdorff paratopological group topology $\tau$ on $G$. Obviously, $w(G, \tau)=\omega$, hence there exists a non-trivial sequence $\{a_{n}: n\in\mathbb{N}\}$ in $(G, \tau)$ converging to $e$. Therefore, the set $K=\{e\}\cup\{a_{n}: n\in\mathbb{N}\}$ is an infinite closed subset of both $(G, \sigma_{1})$ and $(G, \sigma_{2})$ with the infinite complement $G\setminus K$, which is a contradiction.
\end{proof}

Finally, we show that each countable infinite Abelian group does not admit a pair of $T_{1}$-independent Hausdorff paratopological group topologies, which gives a supplement to Theorem~\ref{t2}.

\begin{theorem}\label{t3}
 A countable infinite Abelian group does not admit a pair of $T_{1}$-independent Hausdorff paratopological group topologies.
\end{theorem}

\begin{proof}
Assume that $\tau_{1}$ and $\tau_{2}$ are $T_{1}$-independent Hausdorff paratopological group topologies on the countable infinite group $G$. By the same notation in the proof of Theorem~\ref{t2}, it may assume that the topologies $\sigma_{1}$ and $\sigma_{2}$ in Theorem~\ref{t2} are Hausdorff topological group topologies, see the proof in \cite[Theorem 3.8]{LL2012}. Then it follows from \cite[Theorem 5.4]{AT2017} that $\sigma_{1}\cup\sigma_{2}$ generates a non-discrete Hausdorff paratopological group topology $\tau$ on $G$. From the proof of Theorem~\ref{t2}, we can obtain a contradiction. Hence $G$ does not admit a pair of $T_{1}$-independent Hausdorff paratopological group topologies.
\end{proof}

  \maketitle
\section{Transversality and the submaximal paratopological group topologies}
In this section, we first discuss the submaximal paratopological groups, which plays an important role in the study of the existence a pair of transversal paratopological group topologies on a group.

Maximal paratopological group topologies of $G$ are maximal
elements of the poset $\mathcal{PL}(G)\setminus\{\delta_{G}\}$. An application of the Zorn lemma easily implies
that every element of $\mathcal{PL}(G)\setminus\{\delta_{G}\}$ is contained in a maximal paratopological group topology. The infimum $\mathcal{PM}_{G}$
of all maximal paratopological group topologies on $G$ is the {\it submaximal paratopological group topology} of $G$. The change of the word ``paratopological group'' to ``topological group'' above gives the definition of maximal topological group topology, and the infimum $\mathcal{M}_{G}$
of all maximal group topologies on $G$ is the {\it submaximal group topology} of $G$.

For the convenience, denote by {\bf PTrans} (resp., {\bf Trans}) the class of
paratopological groups (resp., topological groups) that admits a pair of transversal paratopological group topologies (resp., topological group topologies).
Obviously, we have the following proposition.

\begin{proposition}
Let $(G, \tau)$ be a maximal paratopological group. If $G$ is not a topological group, then there exists an open neighborhood $U$ of $e$ such that $U\cap U^{-1}=\{e\}$.
\end{proposition}

What is the relation about $\mathcal{PM}_{G}$ and $\mathcal{M}_{G}$ on group $G$? In order to answer this question, we first show the following theorem. Indeed, it is not surprise that $\mathcal{PM}_{G}$ is a topological group.

\begin{theorem}\label{tttt}
The $\mathcal{PM}_{G}$ is a topological group.
\end{theorem}

\begin{proof}
Let $\tau=\mathcal{PM}_{G}$. Obviously, $(G, \tau^{-1})$ is a paratopological group. It suffices to prove that $(G, \tau^{-1})$ is the infimum of all maximal paratopological group topologies on $G$. Indeed, let $\sigma$ be an arbitrary maximal paratopological group topology on $G$. Then $\sigma^{-1}$ is also a maximal paratopological group topology on $G$. Therefore, we have $\tau\leq \sigma$ and $\tau\leq \sigma^{-1}$, then $\tau^{-1}\leq \sigma$. Suppose that $(G, \tau^{-1})$ is not the infimum of all maximal paratopological group topologies on $G$, then  $\tau^{-1}\lneqq \tau$ by the definition of $\mathcal{PM}_{G}$, which implies that $\tau\lvertneqq \tau^{-1}$. This is a contradiction. Therefore, $(G, \tau^{-1})$ is the infimum of all maximal paratopological group topologies on $G$, that is, $\tau^{-1}=\tau$. Thus $\mathcal{PM}_{G}$ is a topological group.
\end{proof}

The following theorem shows that the topology $\mathcal{PM}_{G}$ is coarser than $\mathcal{M}_{G}$.

\begin{theorem}
For any group $G$, we have $\mathcal{PM}_{G}\leq\mathcal{M}_{G}$.
\end{theorem}

\begin{proof}
It suffices to prove that $\mathcal{PM}_{G}\leq\tau$ for any maximal topological group topology $\tau$ on $G$. Take any a maximal topological group topology $\tau$ on $G$. By the definition of $\mathcal{PM}_{G}$, we have $\mathcal{PM}_{G}\leq\eta$ for any maximal paratopological group topology $\eta$ on $G$ with $\tau\leq\eta$. Since $\tau$ is maximal group topology, it follows from Theorem~\ref{tttt} that $\mathcal{PM}_{G}\leq\tau$.
\end{proof}

However, the following two questions are still unknown for us.

\begin{question}
Does $\mathcal{PM}_{G}=\mathcal{M}_{G}$ hold?
\end{question}

\begin{question}
Let $(G, \tau)$ be a maximal paratopological group. Is $\tau\wedge\tau^{-1}$ a maximal topological group on $G$? What if we assume
the $G$ to be Abelian?
\end{question}

The following theorem is important in the study of the transversality of paratopological groups, which gives a characterization on a non-discrete paratopological group being transversal.

\begin{theorem}\label{tttt1}
Let $(G, \tau)$ be a non-discrete paratopological group topology. Then $(G, \tau)\in$ {\bf PTrans}
iff $\tau\nleqslant \mathcal{PM}_{G}$.
\end{theorem}

\begin{proof}
Assume that $\tau\in$ {\bf PTrans}. Then there exists a paratopological group topology $\sigma$ such that $\tau\vee \sigma$ is discrete. Without loss of generality, we may assume that $\sigma$ is a maximal paratopological group topology on $G$. Then $\mathcal{PM}_{G}\leq\sigma$, thus $\tau\nleqslant \mathcal{PM}_{G}$.

Suppose that $\tau\nleqslant \mathcal{PM}_{G}$. Then there exists a maximal paratopological group topology $\sigma$ on $G$ such that $\tau\nleqslant\sigma$.
Since $\sigma$ is maximal, $\tau\vee \sigma$ is discrete, that is, $\tau\in$ {\bf PTrans}.
\end{proof}

The following corollary implies that $\mathcal{PM}_{G}=\mathcal{M}_{G}$ if and only if $\mathcal{M}_{G}\not\in${\bf PTrans}.

\begin{corollary}
Let $G$ be a group. Then $\mathcal{M}_{G}\in${\bf PTrans} if and only if $\mathcal{PM}_{G}\lvertneqq\mathcal{M}_{G}$.
\end{corollary}

Next we give some sufficient conditions for a paratopological group such that is has a transversal paratopological group topology.

We say that a homomorphism $f: G\rightarrow H$ of abstract groups {\it $\mathcal{PM}$-continuous} (resp., {\it $\mathcal{PM}$-open}) if
$f$ becomes continuous (resp., open) when $G$ and $H$ are equipped with their submaximal
paratopological group topologies.

\begin{proposition}\label{p04}
Let $G$ be a subgroup of $H$ with $G\cdot c_{H}(G)=H$. If $\tau$ is a paratopological group topology on $G$, then one can extend $\tau$ from $G$ to a paratopological group $\sigma$ on $H$ such that $G\in\sigma$. In particular, if $\tau$ is a maximal paratopological group topology on $G$, then $\sigma$ is also a maximal paratopological group topology on $H$.
\end{proposition}

\begin{proof}
Indeed, let $\tau(e)$ be the neighborhood base at $e$ in $G$. It suffices to prove that $\tau(e)$ is a neighborhood base at $e$ in $H$. Only one of the axioms of a base of a paratopological group topology need a special care. Take an any open neighborhood $U\in\tau(e)$ and $h\in H$. We need to find a neighborhood $W\in\tau(e)$ such that $h^{-1}Wh\subset U$. Since $G\cdot c_{H}(G)=H$, there exist $g\in G$ and $c\in c_{H}(G)$ such that $h=gc$. Since $\tau$ is a paratopological group topology on $G$, there exists a $W\in\tau(e)$ such that $g^{-1}Wg\subset U$. Then it follows from $c\in c_{H}(G)$ that $h^{-1}Wh=c^{-1}g^{-1}Wgc\subset c^{-1}Uc=U$.

Moreover, if $\tau$ is a maximal paratopological group topology on $G$, then it easily see that $\sigma$ is also a maximal paratopological group topology on $H$ since $\tau$ is a maximal and $G\in\sigma$.
\end{proof}

\begin{theorem}\label{t02}
Every homomorphism $f: G\rightarrow H$ of abstract groups is $\mathcal{PM}$-continuous
whenever $f(G)\cdot c_{H}(f(G))=H$.
\end{theorem}

\begin{proof}
It suffices to prove that $f: (G, \tau)\rightarrow (H, \mathcal{PM}_{H})$ is continuous for every maximal paratopological group topology $\tau$ on $G$. By Proposition~\ref{p04}, we may assume that $f$ is surjective. Let $N=\mbox{ker}f$. Obviously, $N$ is normal. If $N$ is neither $\tau$-open nor $\tau$-discrete, then it is obvious that one can get a finer non-discrete paratopological group topology on $G$ by taking $N$ open, which is a contradiction with the maximality of $\tau$ on $G$. If $N$ is open in $G$, then $f^{-1}(U)N$ is open for each open subset $U$ in $(H, \mathcal{PM}_{H})$. Now suppose that $N$ is discrete in $G$. Let $\upsilon$ denote the quotient topology on $H$ make $f$ is quotient. Then $f: (G, \tau)\rightarrow (H, \upsilon)$ is local homeomorphism. Moreover, for each open neighborhood $U$ of $e$ in $(G, \tau)$ with $U^{2}\cap N=\{e\}$, $f\upharpoonright_{U}: U\rightarrow f(U)$ is a homeomorphism. Further, this remains valid when $U$ is replaced by any conjugate $U^{a}=a^{-1}Ua$ for $U\in\tau, a\in G$. We claim that $\upsilon$ is maximal. Indeed, suppose not, there exists a paratopological group topology $\theta$ on $H$ such that $\theta$ is strictly finer than $\upsilon$.
Let $$\mathcal{A}=\{f^{-1}(V)\cap U^{a}: V\in \theta(e), V\subset f(U^{a}), U^{2}\cap N=\{e\}, U\in\tau(e)\}.$$
Then $\mathcal{A}$ forms a filter base at $e$ of $G$ which determines a non-discrete paratopological group topology $\eta$ on $G$ with $\eta>\tau$, which is a contradiction. Therefore, $\upsilon$ is maximal, thus $\mathcal{PM}_{H}\leq\upsilon$, which implies that $f$ is $\mathcal{PM}$-continuous.
\end{proof}

\begin{corollary}
Every homomorphism $f: G\rightarrow H$ of abstract Abelian groups is $\mathcal{PM}$-continuous.
\end{corollary}

\begin{theorem}\label{t03}
Let $(H, \tau)$ be a non-discrete paratopological group and let $G$ be an infinite subgroup
of $H$ with $G\cdot c_{H}(G)=H$. Then $H$ admits a transversal paratopological group topology in each of the
following cases:

(1) $G$ admits a transversal paratopological group topology.

(2) $G$ is discrete and non-minimal (i.e., admits a non-discrete paratopological group topology).
\end{theorem}

\begin{proof}
Suppose that $H$ does not admit a transversal paratopological group topology. Then $\tau\leq\mathcal{PM}_{H}$ by Theorem~\ref{tttt1}, hence $\tau\upharpoonright_{G}\leq\mathcal{PM}_{H}\upharpoonright_{G}$. Then it follows from Theorem~\ref{t02} that $\mathcal{PM}_{H}\upharpoonright_{G}\leq \mathcal{PM}_{G}$. Therefore, $\tau\upharpoonright_{G}\leq\mathcal{PM}_{G}$, which shows that $G$ must be discrete by Theorem~\ref{tttt1} and the assumption of our theorem. However, since $\tau\upharpoonright_{G}\leq\mathcal{PM}_{G}$, we conclude that $G$ is a Markov group, hence $G$ is discrete and
minimal, which is a contradiction.
\end{proof}

It is well known that every infinite Abelian group admits a non-discrete group topology, hence it follows from Theorem~\ref{t03} that we have the following corollary.

\begin{corollary}
Suppose that a non-discrete paratopological group $H$ contains an infinite central
subgroup $G$ which is either discrete or admits a transversal paratopological group topology. Then $H$
admits a transversal paratopological group topology.
\end{corollary}

By a similar proof of \cite[Corollary 3.7]{DTY2005}, we also have the following corollary.

\begin{corollary}
Endowed with the Sorgenfrey topologies, the additive groups $\mathbb{Q}, \mathbb{Q}_{p}$ (for $p\in\mathbb{P}$), $\mathbb{R}$, $\mathbb{C}$, the multiplicative groups $\mathbb{R}^{\ast}=\mathbb{R}\setminus\{0\}$, $\mathbb{C}^{\ast}=\mathbb{C}\setminus\{0\}$ and the groups $GL(n,\mathbb{R})$, $GL(n,\mathbb{C})$ admit  transversal paratopological group topologies respectively.
\end{corollary}

Next we give three parallel results about the productivity properties of the submaximal paratopological group topologies and their impact
on the transversability. The reader can see the similar proofs in \cite[Subsection 3.3]{DTY2005}.

\begin{proposition}
For an arbitrary family $\{G_{i}\}i\in I$ of groups and $G=\prod_{i\in I}G_{i}$, one has $\mathcal{PM}_{G}\geq\prod_{i\in I}\mathcal{PM}_{G_{i}}$. If the index set $I$ is finite, this inequality becomes an equality.
\end{proposition}

\begin{theorem}
A product paratopological group $K\times H$ belongs to {\bf PTrans} in each of the following cases:

(a) $H$ admits a transversal paratopological group topology;

(b) $H$ is discrete and non-minimal and $K$ is non-discrete.
\end{theorem}

\begin{theorem}
If $\{G_{i}: i\in I\}$ is a family of Abelian paratopological groups, then $G=\prod_{i\in I}G_{i}\in$ {\bf PTrans} iff there exists $i\in I$ such that $G_{i}\in$ {\bf PTrans}.
\end{theorem}

Finally, we discuss the maximal precompact paratopological group. For a group $G$, we denote by $\mathcal{P}_{G}$ (resp., $\mathcal{PP}_{G}$) the maximal precompact topological group (resp., paratopological group) topology of $G$.

By Theorem~\ref{tttt}, it is easy to see that $\mathcal{PM}_{G}\geq \mathcal{PP}_{G}$ does not hold. Indeed, take the unit circle $\mathbb{T}$ with the half interval topology $\tau$. Then $\tau$ is precompact, hence $\tau\leq \mathcal{PP}_{\mathbb{T}}$. However, $\mathcal{PM}_{\mathbb{T}}$ is a topological group by Theorem~\ref{tttt}, hence $\mathcal{PM}_{\mathbb{T}}\ngeq\mathcal{PP}_{\mathbb{T}}$. Moreover, it follows from Theorem~\ref{tttt1} that $\mathcal{PP}_{\mathbb{T}}$ admits a transversal paratopological group topology, which is different from the fact that no precompact group topology on an infinite group admits a transversal group topology.

Let $G$ be a group. Obviously, $(\mathcal{PP}_{G})^{-1}$ is also a maximal precompact paratopological group on $G$. Let $\mathcal{TP}_{G}$ be the supremum of all topological group topologies which is coarser than $\mathcal{PP}_{G}$ and $(\mathcal{PP}_{G})^{-1}$. Obviously, we have $\mathcal{TP}_{G}\leq\mathcal{M}_{G}$. However, we have the following question.

\begin{question}
Does $\mathcal{PM}_{G}\geq \mathcal{P}_{G}$ or $\mathcal{PM}_{G}\geq \mathcal{TP}_{G}$ hold for every group $G$?
\end{question}

Let $\sharp$ be the class of Abelian group $G$ such that $\mathcal{PM}_{G}$ is precompact. Then we have the following theorem.

\begin{theorem}\label{tt00}
An Abelian group $G$ belongs to $\sharp$ if and only if every non-discrete non-precompact paratopological group topology on $G$ has a transversal one.
\end{theorem}

\begin{proof}
Suppose that every non-discrete non-precompact paratopological group topology on $G$ has a transversal one. By Theorem~\ref{tttt1}, $\mathcal{PM}_{G}$ must be precompact if $(G, \mathcal{PM}_{G})\not\in${\bf PTrans}, thus, $G\in\sharp$.

Assume that $G\in\sharp$, then $\mathcal{PM}_{G}$ is precompact, thus every non-precompact paratopological group topology $\tau$ on $G$ satisfies $\tau\nleq\mathcal{PM}_{G}$. By Theorem~\ref{tttt1}, $(G, \tau)\in${\bf PTrans}.
\end{proof}

\smallskip
\section{The transversality and paratopological group topology determined by PT-sequence}
Note that all groups are Abelian and all paratopological group topologies are supposed to be $T_{1}$ in this section. From the above discussion of the transversality of paratopological groups, it natural to consider the following question.

\begin{question}
What kind of characterization of a paratopological group $(G, \tau)$ such that $(G, \tau)$ admits a transversal topology?
\end{question}

In this section, we use a similar method in \cite{PZ} and introduce the concept of $PT$-filter (in particular, $PT$-sequence) in the class of paratopological groups. We main prove that a paratopological group $G$, endowed with the strongest paratopological group topology in which the sequence $\mathcal{S}$ converges to $0$, does not admit a $T_{1}$-complementary Hausdorff paratopological group topology on $G$. Moreover, we prove that each Hausdorff paratopological group topology on an infinite group determined by
any a $PT$-sequence is transversal. First, we give a definition.

Let $\varphi$ be a family of non-empty subsets of a group $G$. For each sequence $\{U_{n}\}_{n\in\omega}$ of elements of $\varphi$ let $$\sum_{n\in\omega}U_{n}=\bigcup_{n\in\omega}(U_{0}+\ldots+U_{n}).$$
Denote by $\sum(\varphi)$ the family of subsets of the form $\sum_{n\in\omega}U_{n}$, where $U_{n}\in\varphi$ for each $n\in\omega$. For each $U\in\varphi$, put $U^{\ast}=U\cup\{0\}$ and denote $\varphi^{\ast}=\{U^{\ast}: U\in\varphi\}$.

\begin{definition}
A filter $\varphi$ on a group $G$ is called a {\it $PT$-filter} (resp., {\it $T$-filter}) if there exists a paratopological group topology (resp., group topology) $\tau$ in which $\varphi$ converges to $0$. Given any $PT$-filter (resp., $T$-filter) on a group $G$ denote by $P(G|\varphi)$ (resp., $T(G|\varphi)$) the group $G$ endowed with the strongest paratopological group topology (resp., group topology) in which $\varphi$ converges to $0$.
\end{definition}

{\bf Note:} Obviously, if $\varphi$ is a $T$-filter, then $\varphi$ is a $PT$-filter and $T(G|\varphi)\subset P(G|\varphi)$, thus $P(G|\varphi)$ is Hausdorff. However, there exists a $PT$-filter $\varphi$ which is not a $T$-filter. Indeed, take an arbitrary $T_{1}$-paratopological group $(G, \tau)$ which is not Hausdorff. Let $\varphi$ be a filter of neighborhoods of $0$ in $\tau$. Then $(G, \tau)=P(G|\varphi)$. However, $T(G|\varphi)$ does not exist since $(G, \tau)$ is not Hausdorff.

\begin{theorem}\label{t0000}
Let $\varphi$ be a $PT$-filter on group $G$. Then $P(G|\varphi)$ is Hausdorff if and only if $\varphi$ is a $T$-filter on group $G$.
\end{theorem}

\begin{proof}
Obviously, $\varphi$ is a $T$-filter on group $G$, then $P(G|\varphi)$ is Hausdorff since $T(G|\varphi)\subset P(G|\varphi)$. Let $P(G|\varphi)$ be Hausdorff, and let $\delta$ be the strongest group topology on $G$ which is coarser than $P(G|\varphi)$. Then it is easy to see that $\delta$ is Hausdorff since $P(G|\varphi)$ is Hausdorff.
\end{proof}

{\bf Note:} In general, $P(G|\varphi)\neq T(G|\varphi)$ even if $P(G|\varphi)$ is Hausdorff. Indeed, take an arbitrary Hausdorff paratopological group $(G, \tau)$ which is not regular. Let $\varphi$ be a filter of neighborhoods of $0$ in $\tau$. Then  $(G, \tau)=P(G|\varphi)$. However, it follows from Theorem~\ref{t0000} that $\varphi$ is a $T$-filter. Then $P(G|\varphi)\neq T(G|\varphi)$ (that is, $T(G|\varphi)\lvertneqq P(G|\varphi)$) since $P(G|\varphi)$ is not regular.

We give a characterization on a group $G$ such that a filter $\varphi$ is a $PT$-filter, and a representation of neighborhood base at 0 in $P(G|\varphi)$, see Theorem~\ref{t00}. We first show two propositions.

\begin{proposition}\label{p00}
For each filter $\varphi$ on a group $G$ the following statements hold:

\smallskip
($\mbox{a}$) $0\in\bigcap_{U\in\sum(\varphi^{\ast})}U$;

\smallskip
($\mbox{b}$) for any subsets $U, V\in\sum(\varphi^{\ast})$, there is $W\in\sum(\varphi^{\ast})$ such that $W\subset U\cap V$;

 \smallskip
($\mbox{c}$) for each subset $U\in\sum(\varphi^{\ast})$, there is $V\in\sum(\varphi^{\ast})$ such that $V+V\subset U.$
\end{proposition}

\begin{proof}
Obviously, (a) and (b) are true. To prove (c), fix an arbitrary sequence $\{U_{n}\}_{n\in\omega}$ such that $U=\sum_{n\in\omega}U_{n}$. For each $n\in\omega$, let $V_{n}=U_{n}\cap U_{n+1}$. Put $V=\sum_{n\in\omega}V_{n}$. Then $V+V\subset U$.
\end{proof}

\begin{proposition}\label{p01}
Let $(G, \tau)$ be a paratopological group and $\varphi$ a filter on $G$ which converges to $0$ in $\tau$. Then for each neighborhood $U$ of $0$ there is $V\in\sum(\varphi^{\ast})$ such that $V\subset U$.
\end{proposition}

\begin{proof}
By induction, it can find a sequence $\{V_{n}\}_{n\in\omega}$ of neighborhoods of $0$ such that $V_{0}+V_{0}\subset U$ and $V_{n+1}+V_{n+1}\subset V_{n}$ for each $n\in\omega$. Let $V=\sum_{n\in\omega}V_{n}$. Then $V\subset U$ and $V\in\sum(\varphi^{\ast})$.
\end{proof}

\begin{theorem}\label{t00}
A filter $\varphi$ on a group $G$ is a $PT$-filter if and only if $\bigcap\sum(\varphi^{\ast})=\{0\}$. For each $PT$-filter $\varphi$, the family $\sum(\varphi^{\ast})$ is a base of neighborhoods of $0$ for the paratopological group $P(G|\varphi)$.
\end{theorem}

\begin{proof}
Let $\varphi$ be a $PT$-filter and $\tau$ a paratopological group in which $\varphi$ converges to $0$. Since $\tau$ is $T_{1}$, it follows from Proposition~\ref{p01} that $\bigcap\sum(\varphi^{\ast})=\{0\}$.

Assume that $\bigcap\sum(\varphi^{\ast})=\{0\}$. By Proposition~\ref{p00}, $\sum(\varphi^{\ast})$ is a base of neighborhoods of $0$ for some $T_{1}$-paratopological group topology $\eta$ on $G$. By Proposition~\ref{p01}, it has $\eta=P(G|\varphi)$.
\end{proof}

The rest of this section we main concern about a filter being generated by a sequence.
Let $\{a_{n}\}_{n\in\omega}$ be a sequence of elements of $X$. For each $n\in\omega$, put $A_{n}=\{a_{m}: m\geq n\}$. Then $$\varphi=\{F\subset X: A_{n}\subset F\ \mbox{for some}\ n\in\omega\}$$ is a filter with base $\{A_{n}: n\in\omega\}$. We say that filter $\varphi$ is generated by sequence $\{a_{n}\}_{n\in\omega}$. For any $k, m\in\omega$, put $$A(k, m)=\{g_{0}+g_{1}+\ldots+g_{k}: g_{0}, g_{1}, \ldots, g_{k}\in A_{m}^{\ast}\}.$$

\begin{definition}
A sequence $\{a_{n}\}_{n\in\omega}$ of elements of group $G$ is called a {\it $PT$-sequence} if there is a paratopological group topology on $G$ in which $\{a_{n}\}_{n\in\omega}$ converges to $0$. Denote by $P(G|\{a_{n}\}_{n\in\omega})$ the group $G$ endowed with the strongest paratopological group topology in which $\{a_{n}\}_{n\in\omega}$ converges to $0$. We say that a partopological group $\tau$ on $G$ is determined by $PT$-sequence $\{a_{n}\}_{n\in\omega}$ if $(G, \tau)=P(G|\{a_{n}\}_{n\in\omega})$. By a similar way, we can define the $T$-sequence within the class of topological group.
\end{definition}

By Theorem~\ref{t00}, we have the following proposition.

\begin{proposition}\label{p02}
A sequence $\{a_{n}\}_{n\in\omega}$ of a group $G$ is a $PT$-sequence if and only if for each element $g\in G\setminus\{0\}$ there exists a sequence $\{A_{n_{i}}\}_{i\in\omega}$ such that $g\not\in\sum_{i\in\omega}A_{n_{i}}^{\ast}$.
\end{proposition}

The following theorem provides a more convenient criterion of $PT$-sequence.

\begin{theorem}\label{th003}
A sequence $\{a_{n}\}_{n\in\omega}$ of a group $G$ is a $PT$-sequence if and only if for each $k\in\omega$ and for each $g\in G\setminus\{0\}$ there is $m\in\omega$ such that $g\not\in A(k, m)$.
\end{theorem}

\begin{proof}
Necessity. Let $\{a_{n}\}_{n\in\omega}$ be a $PT$-sequence. Fix any $k\in\omega$ and $g\in G\setminus\{0\}$. By Proposition~\ref{p02}, it can find a sequence $\{A_{n_{i}}\}_{i\in\omega}$ such that $g\not\in\sum_{i\in\omega}A_{n_{i}}^{\ast}$. Let $m=\max\{n_{0}, \ldots, n_{k}\}$. Then $A(k, m)\subset\sum_{i\in\omega}A_{n_{i}}^{\ast}$, thus $g\not\in A(k, m)$.

Sufficiency. By Proposition~\ref{p02}, it suffices to prove that for each element $g\in G\setminus\{0\}$ there exists a sequence $\{A_{n_{i}}\}_{i\in\omega}$ such that $g\not\in\sum_{i\in\omega}A_{n_{i}}^{\ast}$. Fix a $g\in G\setminus\{0\}$. Next we shall construct such a sequence $\{A_{n_{i}}\}_{i\in\omega}$ by induction. Take $n_{0}$ such that $g\not\in A_{n_{0}}^{\ast}$. Assume that $n_{0}, \ldots, n_{k}\in\omega$ have been chosen such that $g\not\in \sum_{i\leq k}A_{n_{i}}^{\ast}$, and that a choice of $n_{k+1}$ is impossible. Therefore, there exist sequence $\{g_{0}(i)\}_{i\in\omega}$, $\ldots, \{g_{k+1}(i)\}_{i\in\omega}$ taking their values respectively in $A_{n_{0}}^{\ast}, \ldots, A_{n_{k}}^{\ast}$ so that $g=g_{0}(i)+\ldots+g_{k+1}(i)$, where $\{g_{k+1}(i)\}_{i\in\omega}$ is a subsequence of $\{a_{n}\}_{n\in\omega}$. By passing to a subsequence, we divide the proof into the following two cases.

{\bf Case 1:} $\{g_{j}(i)\}_{i\in\omega}$ is a subsequence of $\{a_{n}\}_{n\in\omega}$ for all $j\leq k$.

By the assumption, it can find $m\in\omega$ such that $g\not\in A(k+1, m)$. However, for any sufficiently large values of $i$ ones have $g=g_{0}(i)+\ldots+g_{k+1}(i)\in A(k+1, m)$, which is a contradiction.

{\bf Case 2:} There exists $j\leq k$ such that $g_{j}(i)=b_{j}$ for all $i\in\omega$.

Let $B=\{j_{1}, \ldots, j_{t}\}$ be the set of all indices with this property. Then we have $$g_{1}=g-b_{j_{1}}-\ldots-b_{j_{t}}=\sum_{s\not\in B}g_{s}(i)$$ for any $i\in\omega$. Becasue $b_{j_{1}}+\ldots+b_{j_{t}}\in\sum_{j\leq k}A_{n_{j}}^{\ast}$ and $g\not\in\sum_{j\leq k}A_{n_{j}}^{\ast}$, it follows that $g_{1}\neq 0$. By the assumption, we can find $m\in\omega$ such that $g_{1}\not\in A(k+1-t, m)$, which is a contradiction with the definition of $g_{1}$.
\end{proof}

{\bf Note:} By \cite[Theorem 2.17]{PZ}, it follows that each infinite group has a nontrivial $PT$-sequence. Moreover, every subsequence of $PT$-sequence is a $PT$-sequence as well. Therefore, there are at least $\mathrm{c}$ distinct $PT$-sequences for any infinite group $G$. Further, it follows from \cite[Theorem 2.1.12]{PZ} that there exists a family of cardinality of $\mathrm{c}$ of pairwise transversal paratopological group topologies which are determined by $PT$-sequences.

\begin{definition}
A subset $U$ of a group $G$ is called a {\it $PT$-subset} if $U$ is a neighborhood of $0$ for some non-discrete paratopological group topology on $G$.
\end{definition}

Given any elements $a_{0}, \ldots, a_{n}$ of group $G$, let $$X(a_{0}, \ldots, a_{n})=\{x_{0}a_{0}+\ldots +x_{n}a_{n}: x_{i}\in\{0, 1, \ldots i+1\}, i\leq n\}.$$It is easy to see that, for each sequence $\{a_{n}\}_{n\in\omega}$ of elements of group $G$, ones have $$\sum_{n\in\omega}A_{n}^{\ast}=\bigcup_{n\in\omega}X(a_{0}, \ldots, a_{n}).$$

\begin{theorem}
For each subset $U$ of group $G$, the following statements are equivalent:

(1) $U$ is a $PT$-subset;

(2) $U$ is a neighborhood of $0$ for some non-discrete paratopological group topology determined by nontrivial $PT$-sequence;

(3) there is a sequence $\{a_{n}\}_{n\in\omega}$ of distinct elements of $G$ so that $\bigcup_{n\in\omega}X(a_{0}, \ldots, a_{n})\subset U$.
\end{theorem}

\begin{proof}
Clearly, (2) $\Rightarrow$ (1) holds. It suffices to prove that (1) $\Rightarrow$ (3) and (3) $\Rightarrow$ (2).

(1) $\Rightarrow$ (3) Let $\tau$ be a non-discrete paratopological group topology on $G$ in which $U$ is a neighborhood of $0$. Choose a sequence $\{U_{n}\}_{n\in\omega}$ of neighborhoods of $0$ such that $U_{0}+U_{0}\subset U$, $U_{n+1}+U_{n+1}\subset U_{n}, n\in\omega$. For each $n\in\omega$, take $a_{n}\in U_{n+1}\setminus U_{n}$. Thus $\bigcup_{n\in\omega}X(a_{0}, \ldots, a_{n})\subset U$.

(3) $\Rightarrow$ (2) By \cite[Lemma 2.1.6]{PZ}, we may assume that there exist a subsequence $\{b_{n}\}_{n\in\omega}$ of $\{a_{n}\}_{n\in\omega}$, an element $g\in G\setminus\{0\}$ and a positive integer $k$ so that $kb_{n}=g$ for all $n\in\omega$. First suppose that $g$ is an element of finite order $m$, hence $mkb_{n}=0$ for each $m\in\omega$. Then we can take a subsequence $\{c_{n}\}_{n\in\omega}$ of the sequence $\{b_{n}\}_{n\in\omega}$ and a positive prime number $p$ such that $\frac{mk}{p}c_{n}\neq0$ for each $n\in\omega$. Therefore, it follows from \cite[Lemma 2.1.6]{PZ} that $\{c_{n}\}_{n\in\omega}$ is a $T$-sequence, thus a $PT$-sequence. Put $l=\frac{km}{p}$. Finally, passing to subsequence we have $X(lc_{0}, \ldots, lc_{n})\subset U$ for each $n\in\omega$. Hence $U$ is a neighborhood of $0$ of the non-discrete paratopological group $P(G|\{lc_{n}\}_{n\in\omega})$.

Thus we may assume that $g$ is an element of infinite order. Put $d_{n}=(n+1)kb_{n}=(n+1)g\neq 0$ for each $n\in\omega$. Obviously, the elements of the sequence $\{d_{n}\}_{n\in\omega}$ are distinct each other. Because $d_{n}\neq 0$ for each $n\in\omega$, it follows from \cite[Lemma 2.1.6]{PZ} that $\{d_{n}\}_{n\in\omega}$ is a $T$-sequence, thus a $PT$-sequence. At last, passing to subsequence we have $X(d_{0}, \ldots, d_{n})\subset U$ for each $n\in\omega$. Hence $U$ is a neighborhood of $0$ of the non-discrete paratopological group $P(G|\{d_{n}\}_{n\in\omega})=P(G|\{(n+1)kb_{n}\}_{n\in\omega})$.
\end{proof}

\begin{corollary}
For each paratopological group $(G, \tau)$ the following statements are equivalent:
\begin{enumerate}
\item $\tau$ admits a transversal paratopological group topology;
\item there exists a neighborhood $U$ of $0$ in $(G, \tau)$ and a $PT$-subset $V$ such that $U\cap V=\{0\}$;
\item there exists a transversal paratopological group topology to $\tau$ determined by sequence.
\end{enumerate}
\end{corollary}

\begin{example}
There exists a $PT$-sequence $\{a_{n}\}_{n\in\omega}$ on the group $\mathbb{Z}$ such that $P(G|\{a_{n}\}_{n\in\omega})$ is a Hausdorff paratopological group which is not a topological group.
\end{example}

\begin{proof}
Let $\{b_{n}\}_{n\in\omega}$ be a sequence of positive integers for all $n\in\omega$, and let $a_{0}=1, a_{1}=b_{1}$ and $a_{n+2}=b_{n+2}a_{n+1}+a_{n}$ for each $n\in\omega$ (such as, the Fibonacci sequence). Then it follows from \cite[Theorem 2.2.7]{PZ} that $\{a_{n}\}_{n\in\omega}$ is a $T$-sequence, thus a $PT$-sequence. Obviously, we have $X(a_{0}, \ldots, a_{n})\subset \mathbb{Z}^{+}$ for each $n\in\omega$, thus $P(G|\{a_{n}\}_{n\in\omega})$ is a Hausdorff paratopological group which is not a topological group.
\end{proof}

Next we prove that $P(G|\{a_{n}\}_{n\in\omega})$ is a sequential space containing a closed copy of $S_{2}$, thus it is not a Fr\'{e}chet-Urysohn space.

\begin{theorem}\label{t001}
For each $PT$-sequence $\{a_{n}\}_{n\in\omega}$ on an any group $G$, the paratopological group $P(G|\{a_{n}\}_{n\in\omega})$ is sequential.
\end{theorem}

\begin{proof}
Assume the contrary. Then there is a sequentially closed subset $F$ such that $0\in \overline{F}\setminus F$. In order to obtain a contradiction, we shall construct inductively a neighborhood $\sum_{m\in\omega}A_{n_{m}}^{\ast}$ of $0$ such that $F\cap \sum_{m\in\omega}A_{n_{m}}^{\ast}=\emptyset$. Because $0\not\in F$ and $F$ is sequentially closed, there exists $n_{0}\in\omega$ satisfying $F\cap A_{n_{0}}^{\ast}=\emptyset$. Assume that $n_{0}, \ldots, n_{m}\in\omega$ have been chosen such that $F\cap \sum_{i\leq m}A_{n_{i}}^{\ast}=\emptyset$. Suppose that, for each $k\in\omega$, we have $$F\cap (A_{n_{0}}^{\ast}+\ldots+A_{n_{m}}^{\ast}+A_{k}^{\ast})\neq\emptyset.$$ Therefore, we can take sequences $\{x_{k}\}_{k\in\omega}$ and $\{y_{k}\}_{k\in\omega}$ such that $x_{k}\in A_{n_{0}}^{\ast}+\ldots+A_{n_{m}}^{\ast}$, $y_{k}\in A_{k}^{\ast}$ and $x_{k}+y_{k}\in F$ for each $k\in\omega$. On passing to subsequence, without loss of generality, we may assume that the sequence $\{x_{k}\}_{k\in\omega}$ converges to some $x\in\sum_{i\leq m}A_{n_{i}}^{\ast}$. Since $\{y_{k}\}_{k\in\omega}$ converges to $0$, the sequence $\{x_{k}+y_{k}\}_{k\in\omega}$ converges to $x$. Since $x_{k}+y_{k}\in F$ for each $k\in\omega$ and $F$ is sequentially closed, it follows that $x\in F$, which is a contradiction with $F\cap \sum_{i\leq m}A_{n_{i}}^{\ast}=\emptyset$. Therefore, there is $n_{m+1}\in\omega$ such that $F\cap \sum_{i\leq m+1}A_{n_{i}}^{\ast}=\emptyset$.
\end{proof}

Note that $A(n, 0)$ is the sum of $n+1$ copies of $A_{0}^{\ast}=\{a_{n}: n\in\omega\}\cup\{0\}$, and $A(-1, 0)=\{0\}$.

\begin{theorem}\label{t006}
If $\{a_{n}\}_{n\in\omega}$ is a nontrivial $T$-sequence in $G$, then the paratopological group $P(G|\{a_{n}\}_{n\in\omega})$ contains a closed subspace homeomorphic to $S_{2}$.
\end{theorem}

\begin{proof}
Without loss of generality, we may assume that $\{a_{n}\}_{n\in\omega}$ is an injective subsequence. Since $\{a_{n}\}_{n\in\omega}$ is a nontrivial $T$-sequence in $G$, we can find a sequence $\{U_{n}: n\in\omega\}$ of neighborhoods of $0$ in $P(G|\{a_{n}\}_{n\in\omega})$ such that the family $\{a_{n}+U_{n}: n\in\omega\}$ is disjoint. We claim that, for every $n\in\omega$, it can take an injective sequence $\{b(n, m)\}_{m\in\omega}$ satisfying the following conditions:

\smallskip
(1) $b(n, m)\in U_{n}$ for sufficiently large $m\in\omega$;

\smallskip
(2) $\{b(n, m)\}_{m\in\omega}$ converges to $0$;

\smallskip
(3) $b(n, m)\in A(n, 0)\setminus A(n-1, 0)$ for each $m\in\omega$.

Indeed, put $b(0, m)=a_{m}$ for each $m\in\omega$. Suppose that, for some $k\in\omega$, the sequence $\{b(k, m)\}_{m\in\omega}$ has been found. Fix $m\in\omega$. Assume that, for each $n\in\omega$, there is $x_{n}\in A_{n}^{\ast}$ such that $b(k, m)+x_{n}\in A(k, 0)$. Choose sequences $\{y_{0}(n)\}_{n\in\omega}, \ldots, \{y_{k}(n)\}_{n\in\omega}$ in $A_{0}^{\ast}$ so that for each $n\in\omega$ we have $$b(k, m)+x_{n}=y_{0}(n)+\ldots+y_{k}(n).$$
Without loss of generality, we may suppose that for each $i\leq k$ the sequence $\{y_{i}(n)\}_{n\in\omega}$ converges to some element $y_{i}\in A_{0}^{\ast}$. In order to obtain a contradiction, we divide the proof into the following two cases.

\smallskip
(a) $y_{i}=0$ for at least one value $i\leq k$.

\smallskip
Because $\{x_{n}\}_{n\in\omega}$ converges to $0$, it follows that $b(k, m)\in A(k-1, 0)$, which is a contradiction with the induction hypothesis.

\smallskip
(b) $y_{i}\neq0$ for each $i\leq k$.

\smallskip
Because $y_{i}(n)\in A_{0}^{\ast}$, we may suppose that the sequences $\{y_{0}(n)\}_{n\in\omega}, \ldots, \{y_{k}(n)\}_{n\in\omega}$ are constant, which implies that $\{x_{n}\}_{n\in\omega}$ is constant. This is a contradiction.

Therefore, there exists $x(m)\in A_{m}^{\ast}$ so that $b(k, m)+x(m)\not\in A(k, 0)$. Let $b(k+1, m)=b(k, m)+x(m)$. Obviously, $b(k+1, m)\in A(k+1, 0)\setminus A(k, 0)$. Clearly, $\{b(k+1, m)\}_{m\in\omega}$ converges to $0$.

Let $c_{0}=0$. For each $n\in\omega$, without loss of generality, we may assume that $b(n, m)\in U_{n}$ for each $m\in\omega$. For each $k, m\in\omega$, let $c(k, m)=a_{k}+b(k, m)$. Put $$C=\{c_{0}\}\cup\{a_{n}: n\in\omega\}\cup\{c(k, m): k, m\in\omega\}.$$ We claim that $C$ is closed and homeomorphic to $S_{2}$.

Indeed, since $\{a_{n}\}_{n\in\omega}$ is a nontrivial $T$-sequence and $T(G|\{a_{n}\}_{n\in\omega})\subset P(G|\{a_{n}\}_{n\in\omega})$, it follows from \cite[Theorem 2.3.8]{PZ} that $C$ is closed in $G$. It suffices to prove that $C$ is homeomorphic to $S_{2}$. Suppose that there exists an injective sequence $\{c(k_{i}, m_{i})\}_{i\in\omega}$ converging to $0$ in $G$.
Obviously, the sequence $\{b(k_{i}, m_{i})\}_{i\in\omega}$ converges to $0$. In order to obtain a contradiction, we shall find a neighborhood $U$ of $0$ such that $b(k_{i}, m_{i})\not\in U$ for each $i\in\omega$. Indeed, we shall construct $U$ inductively in the form $\sum_{m\in\omega}A_{n_{m}}^{\ast}$. Let $n_{0}=0$ and assume that $n_{0}, \ldots, n_{l}\in\omega$ have been taken such that $$\{b(k_{0}, m_{0}), \ldots, b(k_{l}, m_{l})\}\cap (A_{n_{0}}^{\ast}+\ldots+A_{n_{l}}^{\ast})=\emptyset.$$ Suppose that, for each $k\in\omega$, ones have $$\{b(k_{0}, m_{0}), \ldots, b(k_{l}, m_{l}), b(k_{l+1}, m_{l+1})\}\cap (A_{n_{0}}^{\ast}+\ldots+A_{n_{l}}^{\ast}+A_{k}^{\ast})\neq\emptyset.$$ By the compactness of $\sum_{i\leq m}A_{n_{i}}^{\ast}$ in $P(G|\{a_{n}\}_{n\in\omega})$, it follows that $$\{b(k_{0}, m_{0}), \ldots, b(k_{l}, m_{l}), b(k_{l+1}, m_{l+1})\}\cap (A_{n_{0}}^{\ast}+\ldots+A_{n_{l}}^{\ast})\neq\emptyset.$$ Hence $b(k_{l+1}, m_{l+1})\in \sum_{i\leq l}A_{n_{i}}^{\ast}\subset A(l, 0)$, which is a contradiction with (3) of the construction of the sequence $\{b(n, m)\}_{m\in\omega}$. Therefore, there exists $n_{l+1}\in\omega$ such that $$\{b(k_{0}, m_{0}), \ldots, b(k_{l}, m_{l}), b(k_{l+1}, m_{l+1})\}\cap (A_{n_{0}}^{\ast}+\ldots+A_{n_{l}}^{\ast}+A_{n_{l+1}}^{\ast})=\emptyset.$$

Therefore, $C$ is homeomorphic to $S_{2}$.
\end{proof}

Since $S_{2}$ is a sequential non-Fr\'{e}chet-Urysohn space, we have the following corollary.

\begin{corollary}\label{c001}
If $\{a_{n}\}_{n\in\omega}$ is a nontrivial $T$-sequence in $G$, then the paratopological group $P(G|\{a_{n}\}_{n\in\omega})$ is not a Fr\'{e}chet-Urysohn space.
\end{corollary}

By \cite[Theorem 2.3]{L2006}, we have the following corollary.

\begin{corollary}
If $\{a_{n}\}_{n\in\omega}$ is a nontrivial $T$-sequence in $G$, then the paratopological group $P(G|\{a_{n}\}_{n\in\omega})$ contains a closed subspace homeomorphic to $S_{\omega}$.
\end{corollary}

Finally, we discuss the transversality and $T_{1}$-complementary of any infinite group determined by a $PT$-sequence.

\begin{theorem}
Each Hausdorff paratopological group topology on infinite group determined by $PT$-sequence is transversal.
\end{theorem}

\begin{proof}
Let $(G, \tau)$ be determined by a $PT$-sequence $\{a_{n}\}_{n\in\omega}$. By Theorem~\ref{t0000}, $\{a_{n}\}_{n\in\omega}$ is a $T$-sequence. Then it follows from \cite[Corollary 2.4.10]{PZ} that $T(G|\{a_{n}\}_{n\in\omega})$ has a transversal group topology $\eta$. Since $T(G|\{a_{n}\}_{n\in\omega})\subset P(G|\{a_{n}\}_{n\in\omega})=\tau$, it follows that $(G, \tau)$ is a transversal paratopological group topology to $\eta$.
\end{proof}

\begin{proposition}\label{p003}\cite[Proposition 2.4]{MY2002}
Assume that $\sigma$ and $\tau$ are $T_{1}$-independent Hausdorff topologies on a set $G$.
If the topology $\sigma$ is sequential, then the space $(G, \sigma)$ is countably compact and does not
contain non-trivial convergent sequences.
\end{proposition}

\begin{theorem}\label{t007}
Let $\{a_{n}\}_{n\in\mathbb{N}}$ be a $T$-sequence on infinite group $G$. Then $P(G|\{a_{n}\})$ does not admit a $T_{1}$-complementary Hausdorff paratopological group topology on $G$.
\end{theorem}

\begin{proof}
Let $\tau=P(G|\{a_{n}\}_{n\in\mathbb{N}})$. Since $\{a_{n}\}_{n\in\mathbb{N}}$ is a $T$-sequence, $(G, \tau)$ is Hausdorff. Assume the contrary, let $(G, \sigma)$ be a Hausdorff paratopological group topology on $G$ such that $\tau$ and $\sigma$ are $T_{1}$-complementary. Then it follows from Theorem~\ref{t001} and Proposition~\ref{p003} that $(G, \sigma)$ is countably compact and does not
contain non-trivial convergent sequences. We claim that $(G, \sigma)$ is saturated, and then it follows from Corollary~\ref{cooo} that $(G, \sigma)$ does not admit a $T_{1}$-complementary Hausdorff paratopological group topology on $G$, which is a contradiction with our assumption. Indeed, since $\{a_{n}\}_{n\in\mathbb{N}}$ is a $PT$-sequence on infinite group $G$, there exists a neighborhood base $\{U_{n}\}_{n\in\mathbb{N}}$ consisting of countable subsets of $G$. Because $\tau$ and $\sigma$ are $T_{1}$-complementary, then $V+U_{n}$ contains $G$ except finite number points of $G$ for each $n\in\mathbb{N}$ and any open neighborhood $V$ of $e$ in $\sigma$. Therefore, it is easy to see that $\sigma$ is $\omega$-narrow. Hence it follows from \cite[Theorems 2.1 and 2.5]{AS2007} that $(G, \sigma)$ is saturated.
\end{proof}

It is natural to pose the following two questions.

\begin{question}
Let $\{a_{n}\}$ be a $T$-sequence on infinite non-Abelian group $G$. Does $P(G|\{a_{n}\})$ admit a $T_{1}$-complementary Hausdorff paratopological group topology on $G$?
\end{question}

\begin{question}
Let $\{a_{n}\}$ be a $T$-sequence on infinite group $G$. Does $P(G|\{a_{n}\})$ admit a $T_{1}$-independent Hausdorff paratopological group topology on $G$?
\end{question}

\section{$PT$-sequence in countable groups}
In this section, we assume that all groups are countable. We main prove that a countable group $G$ admit a non-discrete $T_{1}$-paratopological group topology if and only if there exits a non-trivial $PT$-sequence $\{a_{n}\}_{n\in\omega}$ of elements of $G$.

\begin{definition}
Let $G$ be an any group and let $\varphi$ be a filter on $G$. We say that $\varphi$ is a {\it multiplicative filter} if, for each subset $U\in\varphi$, there is $V\in\varphi$ such that $V\subset U$ and $VV\subset U$.
\end{definition}

\begin{definition}
For each filter $\varphi$ on a group $G$, there is a maximal multiplicative filter $\phi$ so that $\phi\subset\varphi$. We call $\phi$ a {\it multiplication hull of $\varphi$}.
\end{definition}

\begin{definition}
Let $G$ be a group and let $\{D_{n}\}_{n\omega}$ be a sequence of nonempty subsets of $G$. Let $S_{n}$ be a group of all permutations of the set $\{0, 1, \ldots, n-1\}$. Denote $$\Gamma_{i\leq n}D_{i}=\bigcup_{s\in S_{n+1}}D_{s(0)}D_{s(1)}\ldots D_{s(n)}$$and
$$\Gamma_{n\in\omega}D_{n}=\bigcup_{n\in\omega}\Gamma_{i\leq n}D_{i}.$$
\end{definition}

Let $\varphi$ be a filter on a group $G$. Put $\Gamma(\varphi)=\{\Gamma_{n\in\omega}D_{n}: D_{n}\in\varphi, n\in\omega\}$. Then it follows from \cite[Theorem 3.1.2]{PZ} that the filter $\psi$ with the base $\Gamma(\varphi)$ is a multiplicative hull of the filter $\varphi$.

\begin{definition}
Let $\varphi$ be a filter on a group $G$. Denote by $\varphi^{G}$ the family of all subsets of $G$ of the form $\bigcup_{g\in G}g^{-1}D_{g}g$, where $D_{g}\in\varphi$ for each element $g\in G$.
\end{definition}

\begin{proposition}\label{p002}
Let $\varphi$ be an arbitrary filter on a group $G$. Denote by $\psi$ the filter on $G$ with base $\Gamma((\varphi^{\ast})^{G})$. Then the following statements hold:

\smallskip
(1) $\psi$ is a multiplicative filter;

\smallskip
(2) $aFa^{-1}\in\psi$ for each element $a\in G$ and for each subset $F\in\psi$.
\end{proposition}

\begin{proof}
It suffices to prove (2). Take an arbitrary $F\in\psi$ and choose a sequence $\{F_{n}\}_{n\in\omega}$ of elements of $(\varphi^{\ast})^{G}$ so that $\Gamma_{n\in\omega}F_{n}\subset U$. For each $n\in\omega$, let $F_{n}=\bigcup_{g\in G}g^{-1}D_{g}(n)g$, where $D_{g}(n)\in \varphi^{\ast}$ for each $g\in G$. Put $$U_{n}=\bigcup_{g\in G}g^{-1}D_{ga}(n)g\ \mbox{and}\ U=\Gamma_{n\in\omega}U_{n}.$$ Hence $$a^{-1}Ua=\Gamma_{n\in\omega}a^{-1}U_{n}a=\Gamma_{n\in\omega}F_{n}\subset U,$$which shows that $a^{-1}Ua\in\psi.$
\end{proof}

\begin{theorem}\label{t000}
Let $\varphi$ be a filter on a group $G$ and $\psi$ a filter with base $\Gamma((\varphi^{\ast})^{G})$. Then $\varphi$ is a $PT$-filter if and only if $\bigcap\psi=\{e\}$. Moreover, for each $PT$-filter $\varphi$, the filter $\psi$ is a filter neighborhood of $e$ in the paratopological group topology $P(G|\varphi)$.
\end{theorem}

\begin{proof}
Suppose that $\varphi$ is a $PT$-filter, and assume that $\tau$ is an arbitrary paratopological group topology on $G$ such that $\varphi$ converges to $e$ in $\tau$. Assume that $\tau_{e}$ is the filter of neighborhoods of $e$ in $(G, \tau)$. It suffices to prove that $\tau_{e}\subset\psi$. Indeed, take an arbitrary $U\in\tau_{e}$. From \cite[Lemma 3.1.1]{PZ}, it follows that there is a sequence $\{U_{n}\}$ of elements of $\tau_{e}$ such that $=\Gamma_{n\in\omega}U_{n}\subset U$. For each $g\in G$ and for each $n\in\omega$, take $D_{g}(n)\in\tau_{e}$ so that $g^{-1}D_{g}(n)g\subset U_{n}$, and put $V_{n}=\bigcup_{g\in G}g^{-1}D_{g}(n)g$. Since $\varphi$ is a $PT$-filter, it follows that $\tau_{e}\subset\varphi^{\ast}$, hence $D_{g}(n)\in\varphi^{\ast}$ for each $g\in G$ and for each $n\in\omega$. Therefore, $\Gamma_{n\in\omega}V_{n}\in\psi$. Because $\Gamma_{n\in\omega}V_{n}\subset U$, it follows that $\tau_{e}\subset\psi$. Since the topology $\tau$ is $T_{1}$, we have that $\bigcap\tau_{e}=\{e\}$ and consequently $\bigcap\psi=\{e\}$.

Assume that $\bigcap\psi=\{e\}$. It follows from Proposition~\ref{p002} that $\psi$ is a filter of neighborhoods of $e$ for some $T_{1}$-paratopological group topology on $G$. Then $\varphi$ is a $PT$-filter because $\psi\subset\varphi$. Furthermore, by the arguments of the above paragraph, for each $PT$-filter $\varphi$, the filter $\psi$ is a filter of neighborhoods of $e$ in $P(G|\varphi)$.
\end{proof}

\begin{definition}
Let $G=\{g_{n}: n\in\omega\}$ be a countable group and $\varphi$ an arbitrary filter on $G$. Take any sequence $\{C_{k}\}_{k\in\omega}$ of elements of the filter $\varphi^{\ast}$ and let $$U_{k}=\bigcup_{n\in\omega}g_{n}^{-1}C_{k+n}g_{n}.$$ The subset $\Gamma_{k\in\omega}U_{k}$ is called a {\it standard element} of the family $\Gamma((\varphi^{\ast})^{G})$.
\end{definition}

\begin{lemma}\label{l00}
Let $G$ be a countable group. Then, for each subset $V\in\Gamma((\varphi^{\ast})^{G})$, there is a standard element $U\in\Gamma((\varphi^{\ast})^{G})$ such that $U\subset V$.
\end{lemma}

\begin{proof}
Since $V\in\Gamma((\varphi^{\ast})^{G})$, we can find a sequence $\{V_{k}\}_{k\in\omega}$ of elements of the family $(\varphi^{\ast})^{G}$ so that $V=\Gamma_{k\in\omega}V_{k}$. For each $k\in\omega$, $V_{k}$ has the following form $$V_{k}=\bigcup_{n\in\omega}g_{n}^{-1}C_{k+n}g_{n},$$where $C_{kn}\in\varphi^{\ast}$, and put $$C_{k}=\bigcap_{i+j=k}C_{ij}, U_{k}=\bigcup_{n\in\omega}g_{n}^{-1}C_{k+n}g_{n}.$$Hence $U=\Gamma_{k\in\omega}U_{k}$ is a standard element of the family $\Gamma((\varphi^{\ast})^{G})$ and $U\subset V$.
\end{proof}

By Theorem~\ref{t000} and Lemma~\ref{l00}, we have the following theorem.

\begin{theorem}\label{t004}
Let $G$ be a countable group and $\varphi$ a filter on $G$. Then $\varphi$ is a $PT$-filter if and only if, for each element $g\in G\setminus\{e\}$, there exists a standard element $U\in\Gamma((\varphi^{\ast})^{G})$ such that $g\not\in U$.
\end{theorem}

\begin{definition}
Let $G=\{g_{n}: n\in \omega\}$ and $X=\{x_{ij}: i, j\in\omega\}$ a set of variables. For each $n\in\omega$, put
$$\mathcal{X}_{n}=\{g_{j_{0}}^{-1}x_{i_{0}j_{0}}g_{j_{0}}g_{j_{1}}^{-1}x_{i_{1}j_{1}}g_{j_{1}}\ldots g_{j_{n}}^{-1}x_{i_{n}j_{n}}g_{j_{n}}: x_{ij}\in X, i_{k}+j_{k}\leq n, i_{0}, \ldots, i_{n}\ \mbox{are distinct}\}.$$
We denote these finitely many group words in the alphabet $G\cup X$ by $f(x_{i_{0}j_{0}}, \ldots, x_{i_{n}j_{n}})$.
\end{definition}

\begin{theorem}\label{t008}
Let $G=\{g_{n}: n\in \omega\}$ be a group and let $\varphi$ be a filter on $G$. Then $\varphi$ is a $PT$-filter if and only if, for each element $g\in G\setminus\{e\}$, there exists a sequence $\{C_{n}\}_{n\in\omega}$ of elements of $\varphi^{\ast}$ so that $$g\not\in f(C_{i_{0}+j_{0}}, C_{i_{1}+j_{1}}, \ldots, C_{i_{n}+j_{n}})$$ for each $n\in \omega$ and for each group word $f(x_{i_{0}j_{0}}, \ldots, x_{i_{n}j_{n}})\in \mathcal{X}_{n}$.
\end{theorem}

\begin{proof}
Necessity. Assume that $\varphi$ is a $PT$-filter. Take an arbitrary $g\in G\setminus\{e\}$. It follows from Theorem~\ref{t004} that there exists a standard element $U\in\Gamma((\varphi^{\ast})^{G})$ such that $g\not\in U$. Hence there exists a sequence $\{C_{n}\}_{n\in\omega}$ of elements of $\varphi^{\ast}$ so that $U=\Gamma_{k\in\omega}U_{k}$ is determined by $\{C_{n}\}_{n\in\omega}$. Suppose that there exist $n\in\omega$ and a group word $f(x_{i_{0}j_{0}}, \ldots, x_{i_{n}j_{n}})\in \mathcal{X}_{n}$ so that $$g\in f(C_{i_{0}+j_{0}}, C_{i_{1}+j_{1}}, \ldots, C_{i_{n}+j_{n}}),$$Then $g\in U_{i_{0}}U_{i_{1}}\ldots U_{i_{n}}\subset \Gamma_{k\leq n}U_{k}\subset \Gamma_{k\in\omega}U_{k}$, which is a contradiction. Therefore, $g\not\in f(C_{i_{0}+j_{0}}, C_{i_{1}+j_{1}}, \ldots, C_{i_{n}+j_{n}})$ for each $n\in \omega$ and for each group word $f(x_{i_{0}j_{0}}, \ldots, x_{i_{n}j_{n}})\in \mathcal{X}_{n}$.

Sufficiency. By Theorem~\ref{t004}, it suffices to prove that, for each element $g\in G\setminus\{e\}$, there exists a standard element $U\in\Gamma((\varphi^{\ast})^{G})$ such that $g\not\in U$. Indeed, fix a point $g\in G\setminus\{e\}$. By the assumption, there exists a sequence $\{C_{n}\}_{n\in\omega}$ of elements of $\varphi^{\ast}$ so that $$g\not\in f(C_{i_{0}+j_{0}}, C_{i_{1}+j_{1}}, \ldots, C_{i_{n}+j_{n}})$$ for each $n\in \omega$ and for each group word $f(x_{i_{0}j_{0}}, \ldots, x_{i_{n}j_{n}})\in \mathcal{X}_{n}$. Let $U=\Gamma_{k\in\omega}U_{k}$ is determined by $\{C_{n}\}_{n\in\omega}$. We claim that $g\not\in U$. Suppose not, there exists $n\in\omega$ and a permutation $s\in S_{m+1}$ such that $$g\in U_{s(0)}U_{s(1)}\ldots U_{s(m)}.$$ Let $i_{0}=s(0), \ldots, i_{m}=s(m)$. Since $g\in U_{s(0)}U_{s(1)}\ldots U_{s(m)}$, there exist $j_{0}, \ldots, j_{m}\in\omega$ such that $$g\in g_{j_{0}}^{-1}C_{i_{0}+j_{0}}g_{j_{0}}g_{j_{1}}^{-1}C_{i_{1}+j_{1}}g_{j_{1}}\ldots g_{j_{m}}^{-1}C_{i_{m}+j_{m}}g_{j_{m}}.$$Put $n=\max\{i_{0}+j_{0}, \ldots, i_{m}+j_{m}\}$, and let $$i_{m+1}=m+1, \ldots, i_{n}=n; j_{m+1}= \ldots=j_{n}=0.$$Then $f(x_{i_{0}j_{0}}, \ldots, x_{i_{n}j_{n}})\in \mathcal{X}_{n}$ and $g\in f(C_{i_{0}+j_{0}}, C_{i_{1}+j_{1}}, \ldots, C_{i_{n}+j_{n}})$, which is a contradiction.
\end{proof}

\begin{theorem}\label{t005}
A countable group $G$ admits a non-discrete $T_{1}$-paratopological group topology if and only if there exits a non-trivial $PT$-sequence $\{a_{n}\}_{n\in\omega}$ of elements of $G$.
\end{theorem}

\begin{proof}
By the definition of $PT$-sequence, it suffices to prove the necessity.

Suppose that $\tau$ is a non-discrete $T_{1}$-paratopological group on $G$ and prove the existence of non-trivial $PT$-sequence. Indeed, choose inductively a sequence $\{U_{n}\}_{n\in\omega}$ of neighborhoods of $e$ in $(G, \tau)$ satisfying the following conditions:
$$U_{n+1}U_{n+1}\subset U_{n}, \bigcap_{n\in\omega}U_{n}=\{e\}, g_{0}^{-1}U_{n+1}g_{0}\subset U_{n}, g_{1}^{-1}U_{n+1}g_{1}\subset U_{n}, \ldots, g_{n}^{-1}U_{n+1}g_{n}\subset U_{n}.$$
Then the family $\{U_{n}\}_{n\in\omega}$ can be taken as a base of neighborhoods of $e$ for some first-countable $T_{1}$-paratopological group topology $\delta$ on $G$. Obviously, $\delta$ is non-discrete, hence there exists a non-trivial sequence $\{a_{n}\}_{n\in\omega}$ converging to $e$ in $(G, \delta)$. Thus $\{a_{n}\}_{n\in\omega}$ is a $PT$-sequence.
\end{proof}

\begin{proposition}\label{p005}
For each $PT$-filter $\varphi$ on a group $G$, the family $$\{\bigcup_{n\in\omega}\bigcup_{f\in\mathcal{X}_{n}}f(C_{i_{0}+j_{0}}, C_{i_{1}+j_{1}}, \ldots, C_{i_{n}+j_{n}}): \{C_{n}\}_{n\in\omega}\ \mbox{is a sequence of elements of}\ \varphi^{\ast}\}$$ forms a base of neighborhoods of $e$ in $P(G|\varphi)$.
\end{proposition}

\begin{theorem}\label{t009}
Let $G$ be a countable group and let $\{a_{n}\}_{n\in\omega}$ be a $PT$-sequence in $G$. If $P(G|\{a_{n}\})$ is Hausdorff, then there exists a nontrivial $PT$-sequence $\{b_{n}\}_{n\in\omega}$ in $G$ such that $P(G|\{a_{n}\})$ and $P(G|\{b_{n}\})$ are transversal.
\end{theorem}

\begin{proof}
Since $P(G|\{a_{n}\})$ is Hausdorff, it follows from Theorem~\ref{t0000} that $\{a_{n}\}_{n\in\omega}$is a $T$-sequence in $G$. Then, from \cite[Exercise 3.2.7]{PZ}, there exists a nontrivial $T$-sequence $\{b_{n}\}_{n\in\omega}$ in $G$ such that $T(G|\{a_{n}\})$ and $T(G|\{b_{n}\})$ are transversal. Obviously, $T(G|\{a_{n}\})\leq P(G|\{a_{n}\})$ and $T(G|\{b_{n}\})\leq P(G|\{b_{n}\})$, so that $P(G|\{a_{n}\})$ and $P(G|\{b_{n}\})$ are transversal.
\end{proof}


  \end{document}